\documentclass[a4paper,reqno,10pt]{amsart}
\usepackage{amssymb}
\usepackage{amsmath}
\usepackage{amsthm}
\usepackage{graphicx}
\usepackage{texdraw}
\usepackage{amsfonts}
\usepackage{hyperref}
\usepackage{caption}
\usepackage{color}
\usepackage{epsfig}
\usepackage{calligra}
\usepackage{mathtools}
 \allowdisplaybreaks \numberwithin{equation}{section}
\begingroup
\theoremstyle{plain}
\newtheorem{theorem}{Theorem}[section]
\newtheorem{proposition}[theorem]{Proposition}
\newtheorem{lemma}[theorem]{Lemma}

\theoremstyle{definition}
\newtheorem{definition}[theorem]{Definition}
\newtheorem{remark}[theorem]{Remark}
\newtheorem{example}[theorem]{Example}
\endgroup

\def \de {\mathrm{d}}
\def \e {\epsilon}

\def \Om {\Omega }

\def \A {\mathcal A}
\def \B {\Gamma}
\def \r {r_{_{2N+1}}}

\title[Body of constant width with minimal area in a given annulus]{Body of constant width with minimal area in a given annulus}

\author{A. Henrot, I. Lucardesi}

\begin{document}

\begin{abstract}
In this paper we address the following shape optimization problem: find the planar domain of least area, among the sets with prescribed constant width and inradius. In the literature, the problem is ascribed to Bonnesen, who proposed it in \cite{BF}. In the present work, we give a complete answer to the problem, providing an explicit characterization of optimal sets for every choice of width and inradius. These optimal sets are particular Reuleaux  polygons.
\end{abstract}

\maketitle

\medskip



\section{Introduction}

Bodies of constant width (also named after L. Euler {\it orbiforms}) are fascinating geometric
objects and a huge amount of literature has been devoted to them. We refer to the recent book
\cite{MMO} for a nice presentation of the topic. The fact that many open problems for these objects
remain unsolved, in spite of their simple statement, is probably an element of their popularity.
Among known facts, the famous Blaschke-Lebesgue Theorem asserts that the Reuleaux triangle
minimizes the area among plane bodies of constant width, see \cite{Bla} for the proof of W. Blaschke or \cite{KM} for a more modern exposition and \cite{Le} for the original proof of H. Lebesgue
and \cite{BF} where this proof is reproduced. Let us mention that many
other proofs with very different flavours appeared later, for example \cite{Be}, \cite{CCG},
\cite{Eg}, \cite{Ga} and \cite{Ha}.

A related problem is the following.
For any planar compact set $K$, the set of points between the incircle (the biggest disk contained into $K$) and the circumcircle (the smallest disk containing $K$) is
called the minimal {annulus} associated with $K$  {(in higher dimension, the region between the insphere and circumsphere is called minimal shell).}
For a body of constant width $d$, it is known, see \cite{CG}, that the  {incircle} and the  {circumcircle} are centered at the
same point that we will choose as the origin in all the paper. Moreover, the inradius $r$ and the circumradius $R$
satisfy
\begin{equation}\label{inr}
r+R=d.
\end{equation}
Now, given an annulus $S$ with inner radius $r$ and outer radius $R$ satisfying $r+R=d$
with a fixed $d > 0$, it is natural to try to determine the bodies of
constant width $d$ having $S$ as their
associated minimal annulus and having either maximum or minimum volume.
A.E. Mayer in \cite{Ma1} has given upper
and lower bounds for the areas of plane sets of constant width with prescribed
minimal annulus. In particular Mayer's lower bound 
yields another proof of the Blaschke-Lebesgue
Theorem.  
The maximization problem has been solved by T. Bonnesen and the result is explained in the book
Bonnesen-Fenchel, see \cite{BF},  {pp. 134-135 in the original German edition and p. 143
in the English version}.
For the minimization problem,  {in the same chapter, T. Bonnesen gave a conjecture.
Our result confirms this conjecture and makes it more precise.} In a short paper
\cite{Ma2}, A.E. Mayer  {already} gave some sketch of proof which was not complete.
Let us quote Chakerian-Groemer whose Chapter on Bodies of constant width in the Encyclopedia
of Convexity, see \cite{CG}, is a well-known reference (see also \cite{MMO}):
{\it  "Mayer in \cite{Ma2} gives a sketch of a proof that the minimum
area, for a prescribed annulus, is attained by a certain Reuleaux-type polygon, as
conjectured by Bonnesen, however a detailed proof does not appear to have been
published."} 
This is the motivation of our paper: we wanted to give a correct, complete and modern proof
and describe completely the body of constant width that minimizes the area among bodies
having a given minimal annulus (i.e. bodies having a given inradius).
 {We recall that Reuleaux polygons are the plane convex bodies of constant width $d$ whose boundary consists of a finite (necessarily odd, see e.g. \cite[Section 8.1]{MMO}) number of arcs of circle of radius $d$; when the arcs have all the
same length, the polygons are said to be {\it regular} and, among them, the Reuleaux triangle is the one with 3 arcs (actually, it is the unique Reuleaux polygon with 3 boundary arcs).}

Therefore, in this paper we are concerned with the following problem: determine the optimal shape(s) of
\begin{equation}\label{mr}
\A(r):=\min \Big\{|\Omega|\ :\ \Omega \subset \mathbb R^2,\ \hbox{(convex) body of constant width $w(\Omega)=1$},\ \rho(\Omega)=r\Big\},
\end{equation}
where $\rho(\Omega)$ denotes the inradius. Here, without loss of generality, we have set the width $w$ to be 1 (clearly, for a generic width $t$, the minimum and the minimizers have to be rescaled by $t^2$ and $t$, respectively). Accordingly, the possible values of the inradius $\rho$ run in the closed interval $[1-1/\sqrt{3}, 1/2]$: 
 {the left endpoint, $1-1/\sqrt{3} \sim 0.422$, is the inradius of the Reuleaux triangle, which is well known to be the minimizer of the inradius among bodies of fixed constant width (see \cite{BF}
or \cite{CG}); as for the right endpoint, it is an easy consequence of \eqref{inr}.} 
For the extremal values of $r$, the minimizer is known: on one hand, for $r=1-1/\sqrt{3}$, the optimal shape is the Reuleaux triangle, from Blaschke-Lebesgue theorem; on the other hand, for $r=1/2$, it is clearly the disk of radius $1/2$ which is the only set in the corresponding annulus. 
For generic values of $r$, the existence is straightforward and follows by the direct method of the calculus of variations.

\begin{proposition}[Existence]\label{prop1}
Let $1-1/\sqrt{3}\leq r \leq 1/2$. Then the shape optimization problem $\A(r)$ has a solution.
\end{proposition}

In this paper we give a complete answer to the problem \eqref{mr}, providing an explicit characterization of the minimizers for every $r$. 
Our construction gives, as a by-product, uniqueness  {among Reuleaux polygons}.

In order to state the main result, let us denote by $\r$, $N\in \mathbb N^*$,  {($N\geq 1$)}, the inradius of the regular Reuleaux $(2N+1)$-gon:
$$
\r=1-\frac{1}{2\cos\left(\frac{\pi}{2(2N+1)}\right)}.
$$
The sequence $\{\r\}_N$ is increasing and runs from $1-1/\sqrt{3}$ to $1/2$ (not attained).
\begin{theorem}[Characterization  {of the optimal Reuleaux polygon}]\label{thm} 
Let $1-1/\sqrt{3}\leq r < 1/2$.

If $r=\r$ for some $N\in \mathbb N^*$, then  {an} optimal set of $\A(r)$ is the regular Reuleaux $(2N+1)$-gon. In that case $\A(\r)=(2N+1)F(\r,0)$ where $F$ is the function defined in \eqref{ap1}.

If instead $r_{_{2N-1}}<r<r_{_{2N+1}}$ for some $N\in \mathbb N$, $N\geq 2$, 
setting
$$
\ell:=2 \arctan\left(\sqrt{4(1-r)^2-1}\right),\quad h:=\frac\pi2 - \frac{2N-1}{2}\, \ell,
$$
an optimal set of $\A(r)$ has the following structure:
\begin{itemize}
\item[i)] it is a Reuleaux polygon with $2N+1$ sides, all but one tangent to the incircle;
\item[ii)] the non tangent side has both endpoints on the outercircle and has length 
$$
a:=2\,\arcsin\left((1-r)\sin(h)\right),
$$ 
its two opposite sides have one endpoint on the outercircle and meet at a point in the interior of the annulus; moreover, they both have length 
$$
b:= h + \frac{\ell -a}{2};
$$
\item[iii)] the other $2N-2$ sides are tangent to the incircle, have both endpoints on the outercircle, and have length $\ell$.
\end{itemize}
Moreover, in that case $\A(r)=(2N-2)F(r,0)+F(r,h)$ (with $F$ defined in \eqref{ap1}).
\end{theorem}
 {\begin{remark}
To prove this result our strategy will consist in studying first this shape optimization problem
in the class of Reuleaux polygons and then, to use the density of Reuleaux polygons in the class of bodies of constant width, see e.g. \cite{BF}.
Our construction provides uniqueness of the minimizer {\it in the class of Reuleaux polygon}
and gives the minimal value of the area.
Now, it is not clear whether we have uniqueness in general. To prove that, we should
for example approximate any body of constant width by a sequence of Reuleaux polygons with increasing area lying in the same minimal annulus. 
\end{remark}}
To clarify this result, let us show some picture.

\begin{figure}[h]                                             
\begin{center}                                                
{\includegraphics[height=4truecm] {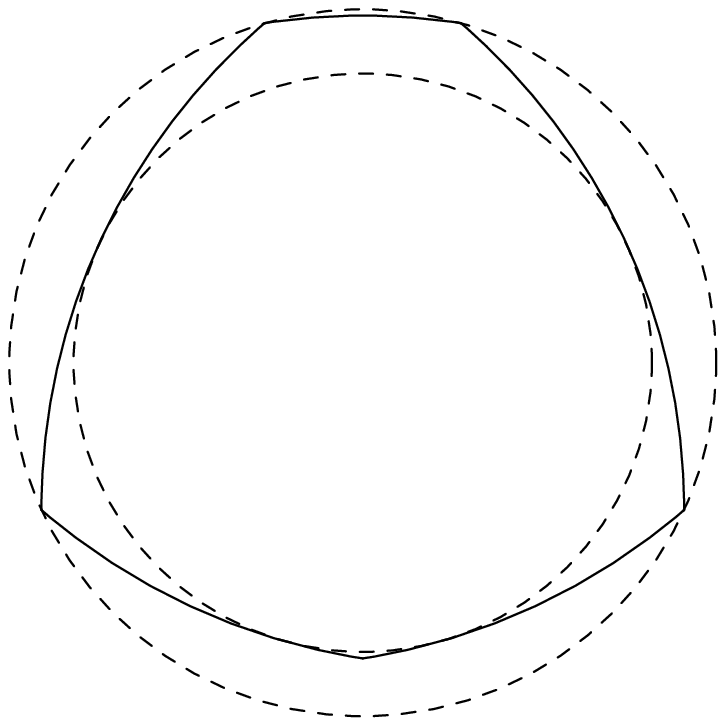}}\quad \quad                                                 
{\includegraphics[height=4truecm] {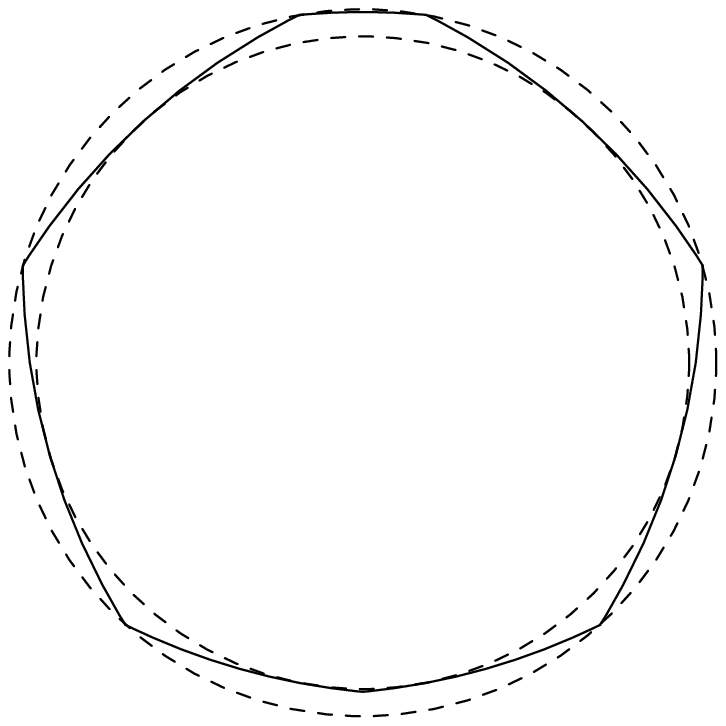}}\quad \quad                                                 
{\includegraphics[height=4truecm] {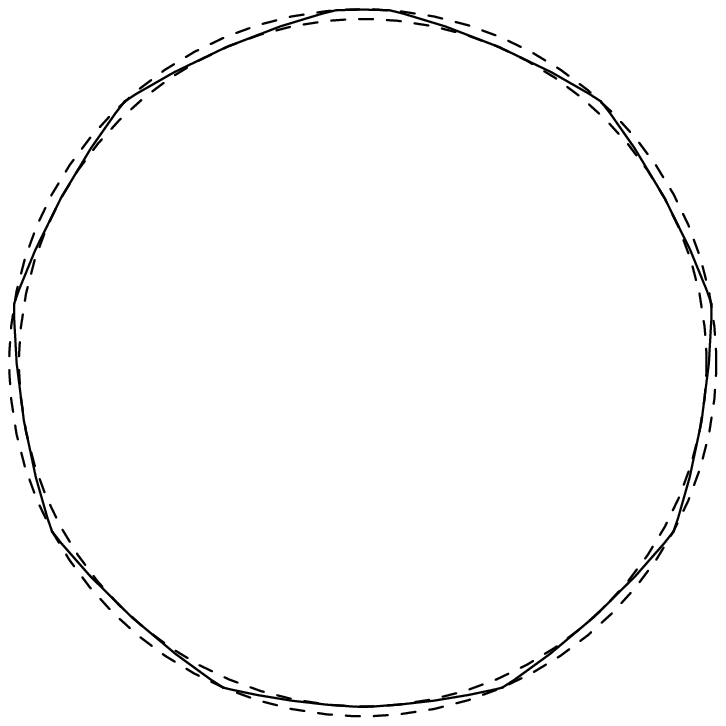}}                                               
\end{center}                                                  
\caption{{\it From left to right: optimal shapes for $r=0.45$, $0.48$, and $0.493$.}}\label{fig-os} 
\end{figure}

Notice that in the limit as $r\to \r$, the lengths $a$, $b$, and $\ell$ all converge to $2 \arctan\left(\sqrt{4(1-\r)^2-1}\right)$, which is the length of the sides of the regular Reuleaux $(2N+1)$-gon. Roughly speaking, in (ii), the interior (to the annulus) point gets closer and closer to the outercircle and the non tangent arc gets closer and closer to the incircle. More precisely, we show the following.
\begin{proposition}[Continuity]\label{prop2}
 {The map $r\mapsto \A(r)$ is continuous in $[1-1/\sqrt{3}, 1/2]$.\\
If we choose in $\mathrm{argmin}\,\A(r)$  the optimal Reuleaux polygon $\Omega(r)$ described in 
Theorem \ref{thm}, then the map $r\mapsto \Omega(r)$ is also continuous with respect to the Hausdsorff distance.}
\end{proposition}
The continuity of $r\mapsto  \Omega(r)$ has to be intended ``up to rigid motion'', namely for every $\epsilon>0$ there exists $\delta>0$ such that
$$
|r_1-r_2|< {\delta} \quad \Rightarrow \quad d_{\mathcal H}\left( \Omega(r_1);\Omega(r_2)\right)< {\epsilon},
$$
for some representative $\Omega(r_i)\in \mathrm{argmin}\,\A(r_i)$, where $d_{\mathcal H}(\,\cdot\,;\,\cdot\,)$ denotes the Hausdorff distance (see, e.g. \cite{H} for the definition).
 {\begin{remark}
The continuity of $r\mapsto \A(r)$ can be obtained either using the explicit formula of $\A(r)$
given above (see the end of the paper) or by a standard argument of $\Gamma$-convergence.
\end{remark}}

We conclude by pointing out that the scope of Theorem \ref{thm} is twofold: on one hand, it gives a complete answer to the Bonnesen's problem; on the other hand, providing a lower bound of the area in terms of geometric quantities, it might prove useful in other shape optimization problems.
 {We use it for example in \cite{HLu} to prove that the Reuleaux triangle maximizes
the Cheeger constant among bodies of constant width}.

\medskip

The plan of the paper is the following. The existence of minimizers (proof of Proposition \ref{prop1}) is given in the next section. As already announced above, in order to characterize the minimizers, we first restrict ourselves to the class of Reuleaux polygons. In this framework, optimal shapes are shown to satisfy an optimality condition, that we call \emph{rigidity} (see Section \ref{sec-rs}). We use as fundamental tool the so called \emph{Blaschke deformations} (see Section \ref{sec-prel}). In the last section we characterize the optimal rigid shapes (Theorem \ref{thm-density}) and we show that actually they are the minimizers of the original problem $\A$ (proof of Theorem \ref{thm}).  {This relies on some analytic argument: the key
point is the concavity property (Proposition \ref{prop5})
of the map $h\mapsto F(r,h)$ that is used to express the area of
each sector. This concavity allows to solve a maximization problem that gives the desired solution.}
The very end of the paper is devoted to the continuity statement (proof of Proposition \ref{prop2}).

\section{Preliminaries and Blaschke deformations}\label{sec-prel}
This section is devoted to some preliminary tools. In the first part, we give the precise definition of  {convex body}, \emph{width}, and \emph{inradius}, and we write the proof of Proposition \ref{prop1}. In the second part, we gather some facts on Reuleaux polygons: more precisely, we recall the notation and the family of deformations introduced by Blaschke in \cite{Bla}, 
see also \cite{KM} for more details, and we write the first order shape derivative with respect to these particular deformations.
 {\begin{definition}
 A {\it convex body} is a compact convex set with nonempty interior.
\end{definition}}
\begin{definition}
Given a  {compact connected} set $\Omega\subset \mathbb R^2$ and a direction $\nu \in \mathbb S^1$, we define the \emph{width} $w_\nu(\Omega)$ of $\Omega$ in direction $\nu$ as the minimal distance of two parallel lines orthogonal to $\nu$ enclosing $\Omega$. We say that $\Omega$ has \emph{constant width} if $w_\nu(\Omega)$ is constant for every choice of $\nu$. In this case, the width is simply denoted by $w(\Omega)$.
The \emph{inradius} of $\Omega$, denoted by $\rho(\Omega)$, is the largest $r$ for which an open disk of radius $r$ is contained into $\Omega$. We also recall the classical Barbier Theorem, see
\cite{CG}: the perimeter of any plane body of constant width $d$ is given by $P(\Om)=\pi d$.
\end{definition}

\medskip

\begin{proof}[Proof of Proposition \ref{prop1}]
By definition the admissible shapes are (strictly) convex and (up to translations) their boundary lie in the closed circular annulus $A:=\overline{B}_{1-r}(0)\setminus B_r(0)$.
If $\Omega_n$, $n\in \mathbb N$, is a minimizing sequence for $\A(r)$, we can extract a subsequence (not relabeled) which, by Blaschke selection theorem, converges for the Hausdorff distance to some convex set $\Omega^*$, whose boundary is in the annulus $A$. 
 {Since the width constraint and the inradius constraint are continuous with respect to the Hausdorff convergence of convex bodies (this is classical and follows from the uniform convergence
of the support functions that is equivalent to the Hausdorff convergence), we conclude that $\Omega^*$ is an admissible shape.} Finally, since the area is also continuous with respect to the Hausdorff convergence for convex domains, $\Omega^*$ is a minimizer for $\A(r)$, concluding the proof.
\end{proof}

\subsection{Reuleaux polygons and Blaschke deformations}

Reuleaux polygons form a particular subclass of constant width sets (here fixed equal to 1), whose boundary is made of an odd number of arcs of circle of radius 1.
The arcs are centered at boundary points, intersection of pairs of arcs. We call such centers \emph{vertexes} and we label them as $P_k$, $k=1,\ldots, 2N+1$, for a suitable $N\in \mathbb N^*$. The arc opposite to $P_k$ is denoted by $\B_k$ and is parametrized by
$$
\B_k=\{P_k + e^{it}\ :\ t\in [\alpha_k, \beta_k]\},
$$
for some pair of angles $\alpha_k$, $\beta_k$. Here, with a slight abuse of notation, $e^{it}$ stands for $(\cos t, \sin t)\in \mathbb R^2$. The subsequent and previous points of $P_k$ are
$$
P_{k+1}=P_k + e^{i\alpha_k}\quad \hbox{and}\quad P_{k-1}=P_{k} + e^{i\beta_k},
$$
respectively. Accordingly, the angles satisfy
$$
\beta_{k+1}= \alpha_k +\pi \,\quad \hbox{mod }2\pi.
$$

The concatenation of the parametrizations of the arcs provides a parametrization of the boundary of the Reuleaux polygon in counterclockwise sense: the order is $\B_{2N+1}$, $\B_{2N-1}$, $\ldots$, $\B_{1}$, $\B_{2N}$, $\B_{2N-2}$, $\ldots$, $\B_2$, namely first the arcs with odd label followed by the arcs with even label. 
Notice that the length of the arc $\B_k$ is $\beta_k-\alpha_k$ and since the perimeter of the Reuleaux polygon is $\pi$ by  {Barbier} Theorem, we have $\sum_k \beta_k - \alpha_k =\pi$.

\begin{remark}
To clarify the notation above, let us see the case of a Reuleaux pentagon. 

\begin{figure}[h]                                             
\begin{center}                                                
{\includegraphics[height=4.5truecm] {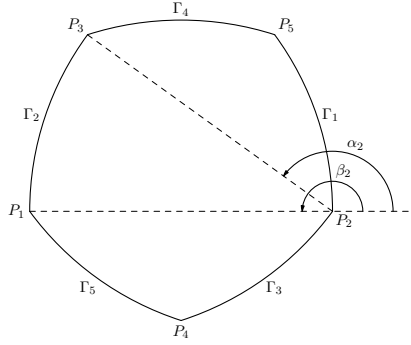}}                                                
\end{center}                                                  
\caption{{\it A Reuleaux pentagon (here, for simplicity, regular).}}\label{fig-penta} 
\end{figure}  

In Figure \ref{fig-penta}, we have chosen, without loss of generality, $\alpha_1=0$, namely the vertex $P_2$ aligned horizontally with $P_1$. Accordingly, the angles are ordered as follows
$$
0=\alpha_1 < \beta_1 < \alpha_4 < \beta_4 < \alpha_2 < \beta_2 < \alpha_5< \beta_5 < \alpha_3< \beta_3<2\pi
$$
and
$$
 \beta_2=\alpha_1+\pi\,,\ \alpha_5 = \beta_1+\pi\,,\  \beta_5=\alpha_4+\pi \,,\  \alpha_3= \beta_4+\pi\,,\  \beta_3=\alpha_2+\pi\,.
$$        
\end{remark}

We now introduce a family of deformations in the class of Reuleaux polygons of width 1, which allow to connect any pair of elements in a continuous way (with respect to the Hausdorff distance), staying in the class. This definition has been introduced by W. Blaschke in \cite{Bla} and analysed by Kupitz-Martini in \cite{KM}.

\begin{definition}\label{def-BD}
Let $\Omega$ be a Reuleaux polygon with $2N+1$ arcs. Let $k$ be one of the indexes in $\{1,\ldots, 2N+1\}$. A \emph{Blaschke deformation} acts moving the point $P_{k}$ on the arc $\B_{k-1}$ increasing or decreasing the arc length. Consequently, the point $P_{k+1}$ moves and the arcs $\B_{k}$, $\B_{k+1}$, and $\B_{k+2}$ are deformed, as in Fig. \ref{fig-def},  {so that the resulting shape is still a Reuleaux polygon.} We say that a Blaschke deformation is \emph{small} if the arc length of $\B_{k-1}$ has changed of $\e\in \mathbb R$,  {infinitesimal parameter. In that case, the number of arcs remains constant.}
\end{definition}

\begin{figure}[h]                                             
\begin{center}                                                
{\includegraphics[height=5.5truecm] {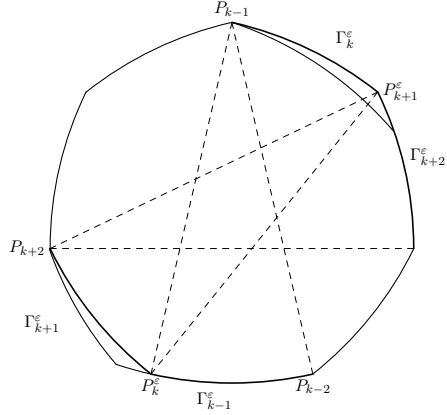}}                                                
\end{center}                                                  
\caption{{\it A Blaschke deformation of a Reuleaux heptagon which moves $P_k$ on $\B_{k-1}$ changing $\alpha_{k-1}$ into $\alpha_{k-1}^\e:=\alpha_{k-1}+\e$, with $\e>0$ small.}}\label{fig-def} 
\end{figure}                

Let us consider a small Blaschke deformation acting on $P_k$ as in Definition \ref{def-BD}, for some small $\e\in \mathbb R$. Let us denote by $\B_i^\e$, $P_i^\e$, $\alpha_i^\e$, and $\beta_i^\e$ the deformed arcs, vertexes, and angles. By definition,
\begin{equation}\label{expa1}
\alpha_{k-1}^\e= \alpha_{k-1}+\e,\quad \beta_k^{\e} = \beta_k + \e.
\end{equation}
The dependence on $\e$ of the other angles is less evident. However, it can be derived by imposing that the transformed configuration is a Reuleaux polygon. Let us determine the first order expansion in $\e$. %In the following, for brevity we use the symbol $\sim$ to denote an error of order $o(|\e|)$. 
The angles $\alpha_k^\e$, $\beta_{k+1}^\e$, and $\beta_{k+2}^\e$  are of the form
\begin{equation}\label{expa2}
\begin{cases}
& \alpha_k^\e = \alpha_k  +  \e \tau +  {o(|\e|)}, \quad \beta_{k+1}^\e = \beta_{k+1} + \e \tau + o(|\e|),
\\
& \alpha_{k+1}^\e = \alpha_{k+1} +  \e \sigma + o(|\e|),  \quad \beta_{k+2}^\e = \beta_{k+2} + \e \sigma + o(|\e|),
\end{cases}
\end{equation}
for some $\sigma$, $\tau\in \mathbb R$. The coefficients $\sigma$ and $\tau$ are uniquely determined by the relation
$$
P_{k+1}^\e = P_k^\e + e^{i \alpha_{k}^\e}=P_{k+2} + e^{i \beta_{k+2}^\e},
$$
which, using the expansions \eqref{expa1} and \eqref{expa2}, easily leads to
\begin{align}
&  e^{i \alpha_{k-1}} + \tau e^{i\alpha_k} = \sigma e^{i\beta_{k+2}}\quad \Longleftrightarrow \quad    e^{i \alpha_{k-1}} - \tau e^{i\beta_{k+1}}  = - \sigma e^{i\alpha_{k+1}} \notag
\\
& \Longleftrightarrow \quad  e^{i (\alpha_{k-1} - \alpha_{k+1})} - \tau e^{i(\beta_{k+1} - \alpha_{k+1})}=-\sigma\notag
\\
& \Longleftrightarrow \quad 
\begin{cases}\label{ab}
\sigma= \sin(\beta_k - \alpha_k) / \sin(\beta_{k+1} - \alpha_{k+1})%=\sin(\delta_k)/\sin(\delta_{k+1}),
\\ \tau= \sin(\alpha_{k-1} - \alpha_{k+1}) / \sin(\beta_{k+1} - \alpha_{k+1}).%=\sin(\delta_{k+1}+\delta_k)/\sin(\delta_{k+1})\,.
\end{cases}
\end{align}

\subsection{Shape derivatives with respect to Blaschke deformations}
In this paragraph we compute the first order shape derivative of the area at a Reuleaux polygon, with respect to a small Blaschke deformation. We recall that, given a one parameter family of small deformations $\Omega_\e$ of $\Omega$, the first order shape derivative of the area at $\Omega$ is nothing but the derivative with respect to $\e$ of the map
$
\e \mapsto |\Omega_\e|
$
evaluated at $\e=0$, namely the limit
$$
\lim_{\e \to 0} \frac{|\Omega_\e|-|\Omega|}{\e}.
$$
 {Note that} in the computation of the first order shape derivative  {the terms} of order $o(\e)$  {in $|\Omega_\e|$} do not play any role.

\begin{proposition}\label{prop4}
Let $\Omega$ be a Reuleaux polygon with angles $\alpha_i$ and $\beta_i$, $i=1,\ldots, 2N+1$.
The first order shape derivative of the area at $\Omega$ with respect to a small Blaschke deformation acting on the point $P_k$ is
$$
\de A_B:=
1-\cos(\beta_k-\alpha_k) - \frac{\sin(\beta_k-\alpha_k)}{\sin(\beta_{k+1}-\alpha_{k+1})}\, \big(1- \cos(\beta_{k+1} - \alpha_{k+1})\big).
$$
that can also be written introducing the lengths $j_k=\beta_k-\alpha_k$ and $j_{k+1}=\beta_{k+1}
-\alpha_{k+1}$ of the arcs 
$\Gamma_k$ and $\Gamma_{k+1}$:
$$
\de A_B=2 \frac{\sin(j_k/2)}{\cos(j_{k+1}/2)}\,\sin\left(\frac{j_k-j_{k+1}}{2}\right).
$$
In particular, the area decreases under a Blaschke deformation if
\begin{itemize}
\item $P_k$ moves on $\Gamma_{k-1}$ in the direct sense ($\e>0$) {\bf and} $j_k<j_{k+1}$,
\item $P_k$ moves on $\Gamma_{k-1}$ in the indirect sense ($\e<0$) {\bf and} $j_k>j_{k+1}$.
\end{itemize}
Moreover, the case where $j_k=j_{k+1}$ corresponds to a local maximum of the area and the area decreases when $P_k$ moves on $\Gamma_{k-1}$ in both senses.
\end{proposition}

\begin{proof}
It is well known (see, e.g. \cite{HP}) that the first order shape derivative of the area at a Lipschitz domain $\Omega$ is a boundary integral which only depends on the normal component of the deformation.
More precisely,  {if $\Phi(\e,\cdot):\mathbb R^2 \to \mathbb R^2$ is a family of diffeomorphisms which map $\Omega$ into $\Omega_\e$, such that $\Phi(0,\cdot)$ is the identity, and such that $\e \mapsto \Phi(\e, \cdot)$ is differentiable at $0$, the first order shape derivative of the area} reads
\begin{equation}\label{boundaryint}
\int_{\partial \Omega} V\cdot n \, \de \mathcal H^1,
\end{equation}
where  {$V(x):=\frac{\partial}{\partial \e}\Phi(\e, x)\lfloor_{\e=0}$. Roughly speaking, for every point $x\in \mathbb R^2$ we have $\Phi(\e,x)=x + \e V(x) + o(|\e|)$. In view of \eqref{boundaryint}, we need to determine the action of $\Phi(\e,\cdot)$ on boundary arcs: for} the Blaschke deformation under study,  {only the arcs} $\B_k$ and $\B_{k+1}$ (see Definition \ref{def-BD})  {are deformed non tangentially, thus in the computation of \eqref{boundaryint} we may disregard all the other arcs.} Using the parametrization $[\alpha_j, \beta_j]\ni t\mapsto P_j + e^{it}$ of $\B_j$, $j=k, k+1$, and noticing that the outer normal vector is $e^{it}$, we immediately have the following simplification:
\begin{equation}\label{shapeder}
\int_{\partial \Omega} V\cdot n \, \de \mathcal H^1 = \int_{\B_k} V \cdot n  \, \de \mathcal H^1 + \int_{\B_{k+1}}V\cdot n \, \de \mathcal H^1  = \int_{\alpha_k}^{\beta_k} V(t) \cdot e^{it} \de t+ \int_{\alpha_{k+1}}^{\beta_{k+1}} V(t)\cdot e^{it}\de t,
\end{equation}
where, for brevity, we have denoted by $V(t)$ the vector $V(P_j + e^{it})$ on the arc $\B_j$, $j=k, k+1$.

 {In order to determine $V$, let us write the $\Phi(\e,\cdot)$ on $\B_{k}$ and $\B_{k+1}$. The arc $\B_{k}^\e$ can be parametrized as follows:
$$
[\alpha_k^\e,\beta_k^\e] \ni t \mapsto P_k^\e + e^{it}
$$
}
or, equivalently, recalling the expansions \eqref{expa1} and \eqref{expa2} of the angles $\alpha_k^\e$ and $\beta_k^\e$, as
 {
$$
[\alpha_k,\beta_k] \ni t \mapsto P_k   + e^{it} + \e ( i e^{i\alpha_{k-1}} + i C_k(t)e^{it}) + o(\e),
$$
}
with
$$
C_k(t)\coloneqq \tau+ (1-\tau)(t-\alpha_k)/(\beta_k-\alpha_k),
$$
and $\tau$ defined in \eqref{ab}. 

Therefore, $V$ acts on the arc $\B_k$ as $V(t)=  i e^{i\alpha_{k-1}} + i C_k(t)e^{it}$. In particular, 
\begin{equation}\label{vn1}
V(t)\cdot e^{it} = \sin(t-\alpha_{k-1}) \quad  \hbox{on }\B_{k}.
\end{equation}
Similarly, using the expansions in \eqref{expa2} of $\alpha_{k+1}^\e$ and $\beta_{k+1}^\e$, and recalling the definition of $\sigma$ in \eqref{ab}, we infer that the arc $\B_{k+1}$ is transformed into $\B_{k+1}^\e$  {parametrized by}
$$
 [\alpha_{k+1}^\e,\beta_{k+1}^\e] \ni t \mapsto  P_{k+1}^\e + e^{it},
$$
 {or equivalently, by}
$$ 
 [\alpha_{k+1},\beta_{k+1}]  \ni t \mapsto P_ {k+1} + e^{it} + \e (i \sigma e^{i \beta_{k+2}} + i C_{k+1}(t)e^{it})
$$
with
$$
C_{k+1}(t)= \sigma + (\tau-\sigma)(t-\alpha_{k+1})/(\beta_{k+1}-\alpha_{k+1}).
$$
Thus, recalling that $\beta_{k+2}=\alpha_{k+1}+\pi$ modulo $2\pi$,
\begin{equation}\label{vn2}
V(t)\cdot e^{it} =-\sigma \sin (t-\alpha_{k+1})\quad  \hbox{on }\B_{k+1}.
\end{equation}

Inserting \eqref{vn1} and \eqref{vn2} into \eqref{shapeder}, developing the integral, we get the first formula. The second follows using elementary trigonometry. The conclusion, in
the case of equality of the lengths, comes from the fact that the derivative becomes negative
in the direct sense when we perform the Blaschke deformation (since $j_{k+1}$ increases
and $j_k$ decreases) and vice-versa.
\end{proof}

\begin{remark} For any non regular Reuleaux polygon, we observe that
we can always choose a Blaschke deformation such that the first derivative of the area
is negative, making the area decrease. This is precisely
the idea used by W. Blaschke in his proof of the Blaschke-Lebesgue Theorem. We can also
make the area increase (for a non regular Reuleaux polygon), which implies the Firey-Sallee
Theorem asserting that the regular Reuleaux polygons maximize the area among Reuleaux polygons
with a fixed number of sides, see \cite{CG}, \cite{KM}.
\end{remark}

\section{Rigid shapes}\label{sec-rs}
We have seen that a Blaschke deformation allows to make the area decrease. Therefore,
for our minimization problem we can concentrate on sets for which no such Blaschke 
deformation is
permitted (because any Blaschke deformation would violate the annulus constraint). This is the sense of the next definition.

\begin{definition}\label{defrigid} Let $r$ be fixed.
We say that a Reuleaux polygon is \emph{rigid} if no Blaschke deformation that
decreases the area can be performed keeping the inradius constraint satisfied. For brevity, since we are searching for minimizers in the class of Reuleaux polygons with width 1 and inradius $r$, we will refer to these particular objects simply as \emph{rigid shapes} or \emph{rigid configurations}.
\end{definition}

A Blaschke deformation is impossible in our class of sets if it moves an arc
inside the incircle (or outside the outercircle) violating the constraint of minimal annulus.
This is why we introduce the following definitions that describe the only possible arcs
such that no deformation is possible.

\begin{definition}
Let be given a Reuleaux polygon of width 1 and inradius $r$. We say that one arc of its boundary is \emph{extremal} if it is tangent to the incircle and both endpoints are on the outercircle.
\end{definition}

\begin{definition}\label{def-clu} Let be given a Reuleaux polygon of width 1 and inradius $r$. We say that three arcs $\B_{k-1}$, $\B_k$, and $\B_{k+1}$ of the boundary form a \emph{cluster} if:\\
the arcs $\B_{k-1}$ and $\B_{k+1}$ are tangent to the incircle, their common point $P_k$ lies in the interior of the annulus $B_{1-r}(0)\setminus \overline{B}_r(0)$, and the other endpoints $P_{k\pm2}$ are on the outercircle.\\ 
%\item[-] the arc $\B_k$ has both endpoints $P_{k\pm1}$ on the outercircle and does not touch the incircle;
%\item[-] the adjacent arcs, $\B_{k\pm2}$ and $\B_{k\pm3}$, are extremal.

Furthermore, we define the \emph{characteristic parameter} $h$ as half of the angle $P_{k+1}\widehat{O}P_{k-1}$. 
These definitions are summarized in Fig. \ref{fig-cluster}.
\end{definition}

\begin{figure}[h]                                             
\begin{center}
{\includegraphics[height=4.5truecm]{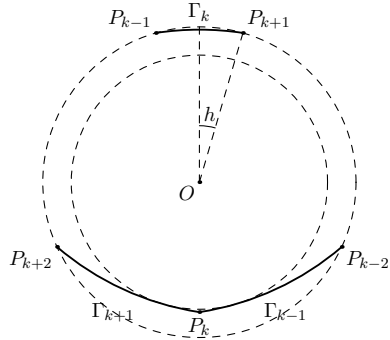}}
\end{center}                                                  
\caption{{\it A triple of arcs $(\B_{k-1}, \B_k, \B_{k+1})$ forming a cluster and the characteristic parameter $h$.}}\label{fig-cluster} 
\end{figure}  

\begin{remark}\label{remark} In the definition of cluster, the arc $\Gamma_k$ can be arbitrarily close to the empty set or to an extremal arc. In the first limit case, we have that $\B_{k-1}$ and $\B_{k+1}$ form a unique arc tangent to the incircle, namely an extremal arc. In the second limit case, $\B_{k-1}$ and $\B_{k+1}$ are a pair of extremal arcs. All in all, extremal arcs (counted individually or in suitable groups of three) can be seen as particular cases of clusters.
At last, let us remark that any cluster has an axis of symmetry:  {by construction, it is the line connecting $P_k$ with the midpoint of the opposite arc $\B_k$ (see also Fig. \ref{fig-cluster}).}
\end{remark}

The fundamental proposition in our approach is the following,  {in which we give a characterization of rigid shapes}. It shows that we can restrict
the study of optimal shapes to Reuleaux polygons having only extremal arcs and clusters.
\begin{proposition}\label{propfund}
The boundary of a rigid shape is made of a finite number (possibly zero) of clusters and of extremal arcs.
\end{proposition}
\begin{proof}
Let us start with some elementary observations.
\begin{itemize}
\item No vertex can lie on the incircle.
\item When a vertex is on the outercircle, its corresponding arc is tangent to the incircle
and this arc goes over the tangent point on both sides. 
\item Conversely, when a vertex is in the interior of the annulus, its corresponding arc is not
tangent to the incircle.
\end{itemize}

Assume that the set $\Omega$ is rigid. 
First of all, let us prove that if a set has two consecutive vertexes, say $P_k$ and $P_{k+1}$
lying in the interior of the annulus, it cannot be rigid. Indeed, in such a case the two arcs
$\Gamma_k$ and $\Gamma_{k+1}$ are not tangent
and therefore the Blaschke deformation described in Definition
\ref{def-BD} is admissible in both senses ($\e>0$ or $\e<0$) without violating the annulus
constraint. Now, following Proposition \ref{prop4} we see that such a deformation will decrease the area by choosing $\e>0$ if $j_k\leq j_{k+1}$ or $\e<0$ if $j_k\geq j_{k+1}$.

Now let us consider a point $P_k$ lying in the interior of the annulus with its two
opposite points $P_{k-1}$ and $P_{k+1}$ on the outercircle. We want to prove
that these three points belong to a cluster, namely that $P_{k-2}$ and $P_{k+2}$ are
on the outercircle. Let us assume, for a contradiction, that $P_{k+2}$ is in the interior
of the annulus (it will obviously be the same proof with $P_{k-2}$). According to the beginning of the
proof, necessarily $P_{k+3}$ has to be on the outercircle. 

In that case, two particular admissible Blaschke deformations can be considered:
\begin{itemize}
\item Move $P_{k+1}$ on $\Gamma_k$ in the direct sense (in the direction of $P_{k-1}$).
\item Move $P_{k+1}$ on $\Gamma_{k+2}$ in the indirect sense (in the direction of $P_{k+3}$).
\end{itemize}
According to Proposition \ref{prop4}, the area will decrease for the first deformation
as soon as $j_{k+1} \leq j_{k+2}$, while it  will decrease for the second deformation
as soon as $j_{k+1} \leq j_{k}$. Therefore, we obtain the conclusion (this configuration
is not rigid) if we can prove $j_{k+1} \leq \max(j_{k},j_{k+2})$. This claim is proved in the next Lemma.
\end{proof}

\begin{lemma}
Assume that $P_k$ and $P_{k+2}$ lie in the interior of the annulus, and that $P_{k-1},P_{k+1},P_{k+3}$ lie on the outercircle. Then the lengths $j_k$, $j_{k+1}$, $j_{k+2}$ of the arcs $\Gamma_k$, $\Gamma_{k+1}$, $\Gamma_{k+2}$, satisfy $j_{k+1} \leq \max(j_{k},j_{k+2})$. 
\end{lemma}
\begin{proof}
Let us introduce the two characteristic parameters $h_k$ and $h_{k+2}$ as half of the angles $P_{k+1}\widehat{O}P_{k-1}$ and $P_{k+3}\widehat{O}P_{k+1}$, see Fig. \ref{fig-lemma}.
\begin{figure}[h]                                             
\begin{center}                                                
{\includegraphics[height=4.7truecm] {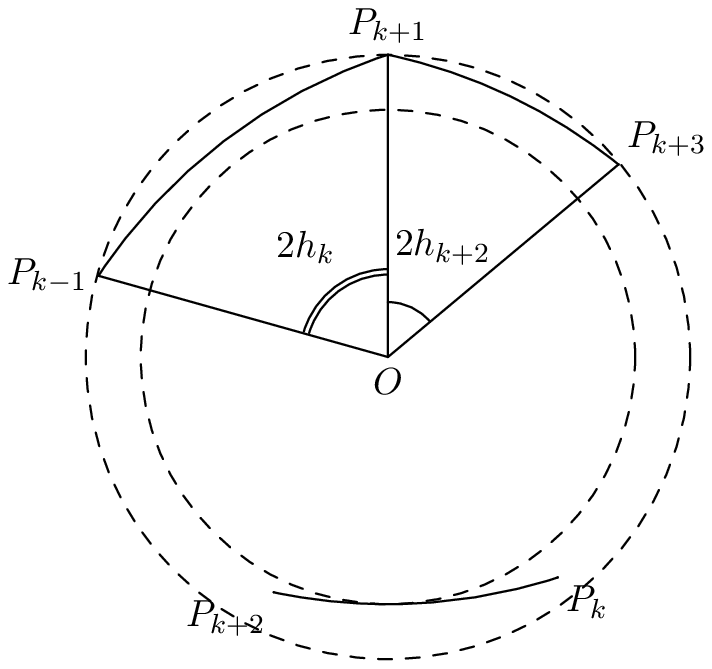}}\quad\quad \quad  {\includegraphics[height=4.7truecm] {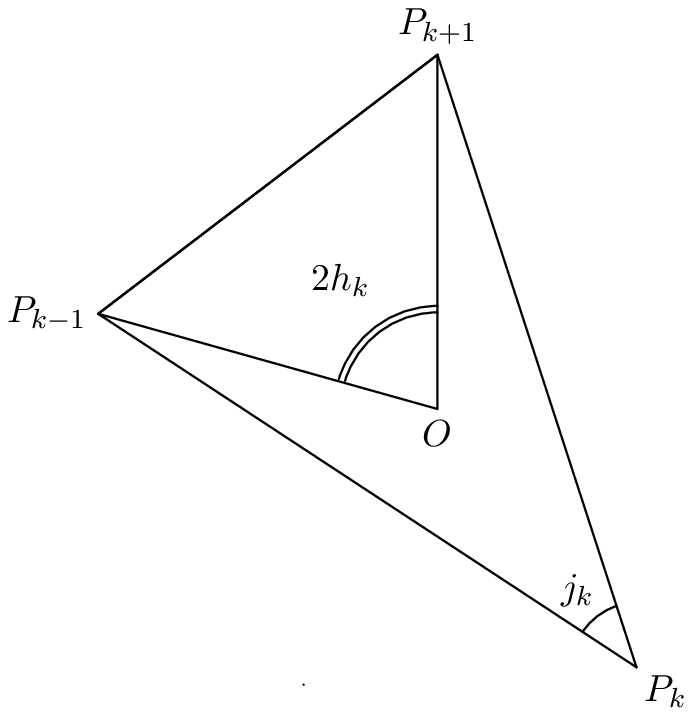}}                                             
\end{center}                                                  
\caption{{\it The parameters $h_{k}$ and $h_{k+2}$.}}\label{fig-lemma} 
\end{figure}                                                  
Elementary trigonometry provides the following relations with the corresponding lengths $j_k$ and $j_{k+2}$:
$$(1-r)\sin h_k=\sin(j_k/2), \quad (1-r)\sin h_{k+2}=\sin(j_{k+2}/2).$$
Now let us write the angle (or length) $j_{k+1}$ as
$$j_{k+1}=P_{k+2}\widehat{P_{k+1}}P_{k}=P_{k+2}\widehat{P_{k+1}}O+O\widehat{P_{k+1}}P_{k}.$$
In the triangles $P_{k+1}OP_{k+2}$ and $P_{k+1}OP_k$ we get the relations
$$P_{k+2}\widehat{P_{k+1}}O=h_{k+2}-\frac{j_{k+2}}{2}, \quad
O\widehat{P_{k+1}}P_{k}= h_{k}-\frac{j_{k}}{2}.$$
Therefore
$$j_{k+1}=h_k+h_{k+2} - \frac{j_k+j_{k+2}}{2}.$$
Now the lengths $j_k, j_{k+2}$ are less than the length of an extremal arc given by 
$\ell=2\arctan\left(\sqrt{4(1-r)^2 - 1}\right)$ (see Proposition \ref{prop-clu}).
We get the thesis if we can prove that for two positive numbers $x,y \in [0,\ell]$ and for $r\in [1-1/\sqrt{3},1/2]$, we have
\begin{equation}\label{ineq1}
\arcsin\left(\frac{\sin(x/2)}{1-r}\right) + \arcsin\left(\frac{\sin(y/2)}{1-r}\right)
- \frac{x+y}{2} \leq \max(x,y).
\end{equation}
Without loss of generality, by symmetry, we can assume $y\geq x$, so that the right-hand side in \eqref{ineq1} is $y$. Let us introduce the function
$$G(x,y):=\arcsin\left(\frac{\sin(x/2)}{1-r}\right) + \arcsin\left(\frac{\sin(y/2)}{1-r}\right)
- \frac{x+3y}{2} .$$
We have
$$\frac{\partial G}{\partial y} = \frac{1}{2(1-r)} 
\frac{\cos(y/2)}{\sqrt{1-\frac{\sin^2(y/2)}{(1-r)^2}}} - \frac{3}{2}.$$
Since the function
$$c\mapsto \frac{c}{\sqrt{1-\frac{1-c^2}{(1-r)^2}}}$$
is decreasing (its derivative has the sign of $1-1/(1-r)^2$), the maximum value of the
derivative
$\frac{\partial G}{\partial y}$ is obtained for $y=\ell$.
This implies
$$\frac{\partial G}{\partial y} \leq \frac{1}{2(1-r)} 
\frac{\cos(\ell/2)}{\sqrt{1-\frac{\sin^2(\ell/2)}{(1-r)^2}}} - \frac{3}{2} = \frac{1}{2-4(1-r)^2} - \frac{3}{2} \leq 0,
$$
where we have used the expression $\cos(\ell/2)=1/(2(1-r))$ and the bound $(1-r)^2 \leq 1/3$. Therefore, $y\mapsto G(x,y)$ is decreasing
and its maximum on the triangle $0\leq x\leq y \leq\ell$ is on the line $x=y$.
Exactly in the same way, it is immediate to check that $x\mapsto G(x,x)$ is decreasing, thus $G(x,y)\leq G(0,0)=0$, proving the lemma.
\end{proof}

\begin{example}\label{ex} Regular Reuleaux polygons are clearly rigid shapes. For a generic $r\notin \{\r\}_N$, many different rigid configurations can be constructed as we will see below. 
For example, there is one (up to rotations) rigid configuration with one
cluster, since in that case the parameter $h$ is fixed (see \eqref{uniqueh} below). When $r$ is large enough, we can
find a continuous family of rigid shapes with two clusters. They are
characterized by an arbitrary pair
of parameters $h_1,h_2$ such that their sum $h_1+h_2$ is fixed. And similarly
for rigid shapes with more clusters. Actually, as shown by Proposition \ref{prop-clu} below,
the lengths of arcs in a cluster are completely characterized by the parameter $h$, moreover, the constraint that the sum of all lengths is $\pi$ fixes the sum of these parameters.

\begin{figure}[h]                                             
\begin{center}
{\includegraphics[height=4.5truecm]{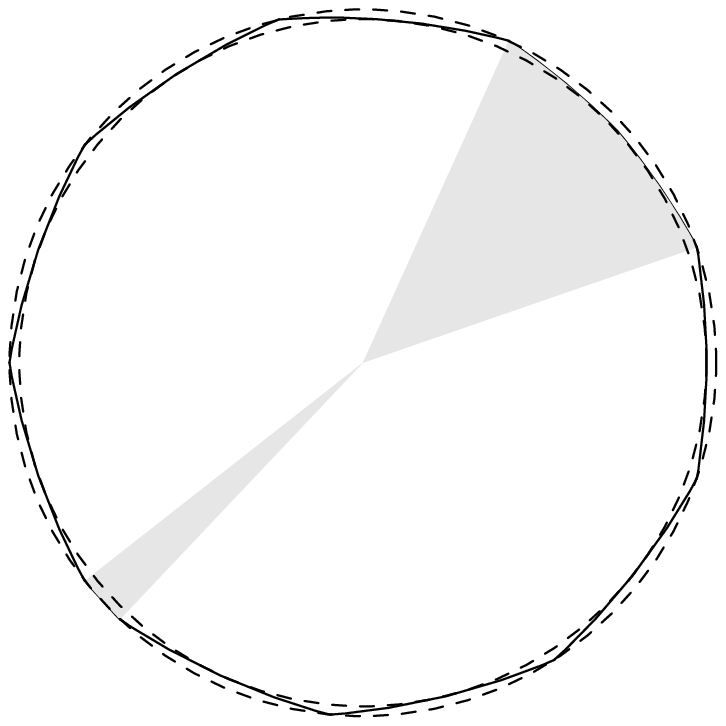}}\quad \quad {\includegraphics[height=4.5truecm]{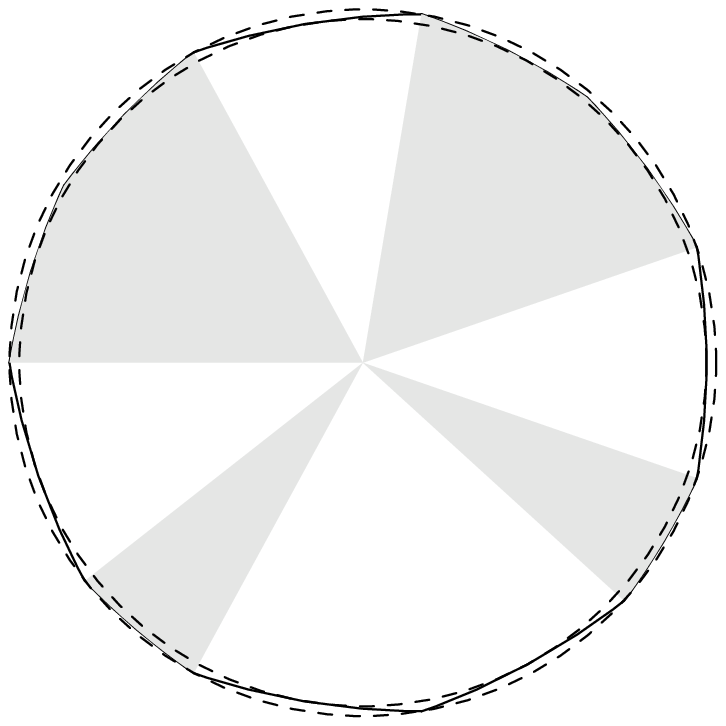}}
\end{center}                                                  
\caption{{\it Two rigid configurations for $r=0.493$. On the left, the one with a single cluster, on the right, one with two clusters.}}\label{fig-clu} 
\end{figure}  

\end{example}

In the next proposition we show that the length of an extremal arc is uniquely determined by $r$, whereas that of a cluster can be expressed as a function of $h$ (which is,  {on the other hand}, not uniquely determined by $r$, see also Example \ref{ex}).

\begin{proposition}\label{prop-clu} Let $r$ be fixed. The length of an extremal arc is
\begin{equation}\label{statement-l}
\ell(r):=2 \arctan\left(\sqrt{4(1-r)^2-1}\right).
\end{equation}
Let $(\B_{k-1}, \B_k, \B_{k+1})$ form a cluster of parameter $h$. Then the length of the arc $\B_k$ is
\begin{equation}\label{statement-a}
a(r,h):=2 \arcsin((1-r)\sin(h)),
\end{equation}
and the length of the opposite arcs $\B_{k\pm 1}$ is
\begin{equation}\label{statement-b}
b(r,h):=h + \frac{\ell(r)-a(r,h)}{2}.
\end{equation}
Moreover, $b(r,h)\geq a(r,h)$.
\end{proposition}

\begin{proof}
Throughout the proof we omit the dependence on $r$ and $h$, which are fixed.

Let $\ell$ denote the length of an extremal arc $\B$ with opposite point $P$ and endpoints $Q$ and $R$. The triangle $POQ$ is isosceles, with base of length 1, legs of length $1-r$, and base angle $\ell/2$, see also Fig. \ref{fig-ellofr}. Therefore $\cos(\ell/2)=1/(2(1-r))$, which gives \eqref{statement-l}.
\begin{figure}[h]                                             
\begin{center}                                                
{\includegraphics[height=4.5truecm] {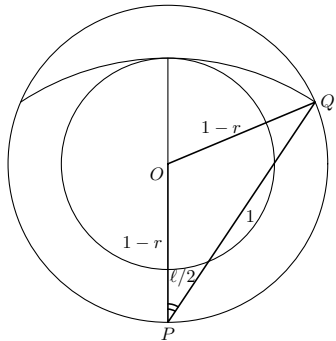}}                                                
\end{center}                                                  
\caption{{\it Computation of $\ell$.}}\label{fig-ellofr} 
\end{figure}                                                  

Let us now consider a cluster of parameter $h$. Without loss of generality, the involved vertexes are $P_1, \ldots, P_5$, oriented in such a way that the parameter $h$ is the angle between the vertical line through $O$ and the segment $OP_4$, see Fig. \ref{fig-cluster5}-left.
According to this notation, we have to determine the length $a$ of $\B_3$, and the length $b$ of $\B_4$ and $\B_2$. Let us consider the triangle $P_3OP_4$, see Fig. \ref{fig-cluster5}-right: the side $P_3P_4$ has length 1 and its opposite angle is $\pi-h$; similarly, the side $OP_4$ has length $1-r$ and its opposite angle is $a/2$; therefore $a$ is determined by the relation $\sin(a/2) = (1-r) \sin h$, which implies \eqref{statement-a}.
\begin{figure}[h]                                             
\begin{center}                                                
{\includegraphics[height=4.7truecm] {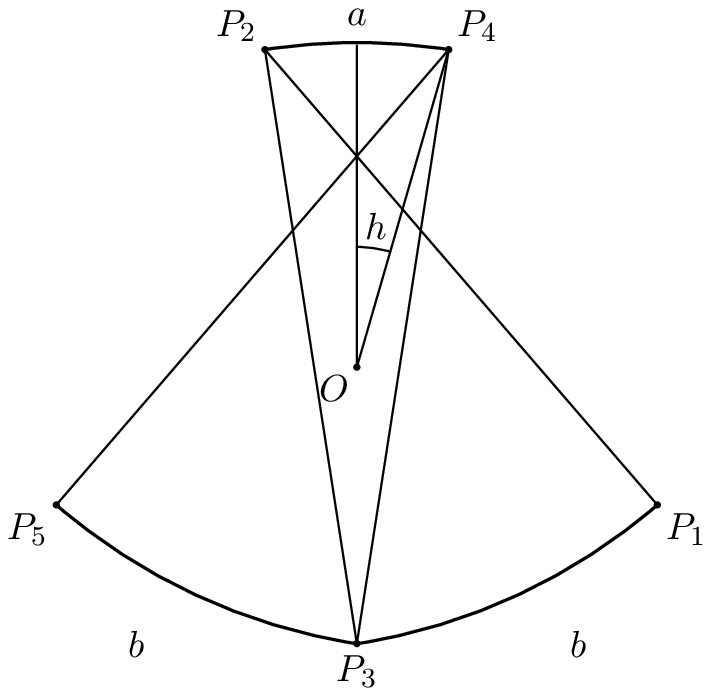}}\quad\quad \quad  {\includegraphics[height=4.7truecm] {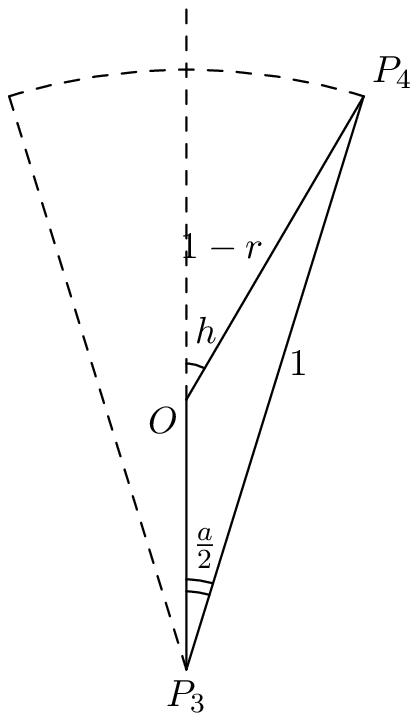}}                                             
\end{center}                                                  
\caption{{\it Left: cluster configuration under study. Right: computation of $a$.}}\label{fig-cluster5} 
\end{figure}                                                  
Let us now compute $b$. It is the sum of two angles: $O\widehat{P_4}P_5$ and $O\widehat{P_4}P_3$. The former is $\ell/2$, since it is the base angle of an isosceles triangle with basis 1 and legs $1-r$ (see also Fig. \ref{fig-ellofr}). The latter can be determined by difference and equals $h-a/2$ (see also Fig. \ref{fig-cluster5}-right). Summing up, we get \eqref{statement-b}.

\medskip
Finally to prove that $b> a$, we have to study the function 
$G: h\mapsto 3 a(h)/2 -h-\ell/2$ for $h\in (0,\ell)$ that are the possible values for
the parameter $h$. Its derivative is given by
$$G'(h)=\frac{3(1-r)\cos h}{\sqrt{1-(1-r)^2\sin^2h}} - 1.$$
Since the function $c\mapsto c/\sqrt{1-(1-r)^2(1-c^2)}$ is increasing (its derivative
has the sign of $1-(1-r)^2$), we see that $G'h)\geq G'(0)=3(1-r)-1>0$ since $r\leq 1/2$.
Thus, $G$ is increasing. Finally $G(\ell)=0$ because $\arcsin((1-r)\sin\ell)=\arcsin(\sin(\ell/2))=\ell/2$, therefore $G(h)< 0 \Leftrightarrow b(h) > a(h)$ for $h<\ell$.
\end{proof}

By definition and in view of the last proposition, the parameter associated to a cluster is between $0$ and $\ell(r)$. Another constraint comes from the fact that the perimeter of Reuleaux polygons of width 1 is $\pi$: given a rigid configuration of inradius $r$, $2N+1$ sides, and $m$ clusters of parameters $h_1, \ldots h_m \in (0,\ell(r))$, there holds
$$
\sum_{i=1}^m [a(r,h_i) + 2 b(r,h_i)] + (2N+1-3m)\ell(r)=\pi,
$$
where $a$ and $b$ are the functions defined above. Recalling the relation \eqref{statement-b}, we get
\begin{equation}\label{nc}
2 \sum_{i=1}^m h_i  + (2N+1-2m)\ell(r)= \pi.
\end{equation}

\begin{remark}\label{rem-cluster} 
The constraint \eqref{nc} can be written in a more general form, allowing the parameters $h_i$ to take also the values $0$ and $\ell(r)$. Indeed, as already noticed in Remark \ref{remark}, extremal arcs can be seen as degenerate cases of clusters: when $h=0$ the arc $\B_k$ reduces to a point whereas the two opposite sides $\B_{k-1}$ and $\B_k$ form a unique arc of length $\ell(r)$; when $h=\ell(r)$, the triple $(\B_{k-1},\B_k, \B_{k+1})$ is of extremal arcs. In both cases, the formulas above for $a$, $b$, and perimeter are still valid. Therefore, every rigid shape can be described in terms of a collection of parameters $h_i$, $i=1,\ldots, \tilde{m}$, varying in the closed interval $[0,\ell(r)]$. The necessary condition \eqref{nc} reads
\begin{equation}\label{nc2}
\sum_{i=1}^{\tilde m} [2 h_i + \ell(r)]= \pi.
\end{equation}
\end{remark}

In the remaining part of the section, we define a family of rigid shapes $\{\Omega(r)\}_r$, whose optimality for $\A(r)$ will be proven in the next section.
\begin{definition}\label{def-optimals} Let $r\in [1-1/\sqrt{3}, 1/2]$. We define
\begin{equation}\label{def-N}
N(r):=\left \lceil{  \frac{\pi}{2\ell(r)}-\frac12 }\right \rceil,
\end{equation}
where $\left \lceil{x}\right \rceil$ denotes the ceiling function of $x$, namely the least integer greater than or equal to $x$. This is the inverse of the function which associates to $r$ the unique $N\in \mathbb N^*$, such that $r\in (r_{2N-1}, r_{2N+1}]$.
We define $\Omega(r)$ as the regular $(2N(r)+1)$-gon if $r=r_{2N(r)+1}$, and as the unique rigid shape with $2N(r)+1$ sides and only one cluster. In this last case, the parameter $h$ associated to the cluster is uniquely determined by $r$, thus we may denote it by $h(r)$: in view of \eqref{nc}, it reads
\begin{equation}\label{uniqueh}
h(r):= \frac{\pi - (2N(r)-1)\ell(r)}{2}.
\end{equation}
\end{definition}

\section{Proof of Theorem \ref{thm} and Proposition \ref{prop2}}\label{sec-proof}

This section is devoted to the proofs of the main results. As announced in the Introduction, as a first step we address the problem $\A_N$, $N\in \mathbb N^*$, of area minimization restricted to the class of Reuleaux polygons with at most $2N+1$ sides. We will prove the following.

\begin{theorem}\label{thm-density} 
The area minimization problem restricted to the family of Reuleaux polygons with at most $2N+1$ sides, $N\in \mathbb N^*$, has the following solution:  {if $N<N(r)$, then there is no admissible shape for $\A_N(r)$; otherwise, if $N\geq N(r)$, the unique (up to rigid motion) minimizer of $\A_N(r)$ is $\Omega(r)$,}
where $N(r)$ and $\Omega(r)$ are the function and the shape introduced in Definition \ref{def-optimals}.
\end{theorem}
In order to prove Theorem \ref{thm-density}, we need to compute the area of a rigid shape. To this aim, we split a shape with $M$ sides into $M$ subdomains, by connecting with straight segments the origin to the vertexes. As previously, the origin is put at the center
of the minimal annulus. The elements of this partition can be regrouped as triples of subdomains associated to clusters and subdomains associated to extremal arcs.  {Examples of triples of subdomains associated with clusters are the gray regions in Fig. \ref{fig-clu} : on the left, 1 triple; on the right, 2 triples.} In the next lemma we provide a formula for the areas of these subdomains.

\begin{lemma}\label{lem-F} Let $r$ be the inradius. Then the area of a triple of subdomains associated to a cluster of parameter $h$ is 
\begin{equation}\label{ap1}
F(r,h):=(1-r)^2 \sin h \cos h +\frac{a-\sin a}{2}+(1-r) \big(\cos(a/2)-(1-r)\cos h\big) \sin(h+\ell) +b-\sin b,
\end{equation}
where $\ell=\ell(r)$, $a=a(r,h)$ and $b=b(r,h)$ are the functions introduced in \eqref{statement-l}, \eqref{statement-a}, and \eqref{statement-b}, respectively.
\end{lemma}
\begin{proof} The area $F(r,h)$ is the sum of two terms: $F(r,h)=|A|+2|B|$, where $A$ is the subdomain with boundary arc of length $a$ and $B$ is one of the two subdomains with boundary arc of length $b$ (which clearly have the same area).
Each of them can be furtherer decomposed as a triangle of the form $OP_jP_{j+2}$ and a portion of disk. 
For the subdomain $A$, the triangle is isosceles: the two sides which meet at $O$ have length $1-r$ and meet with an angle of $2h$, therefore the area is 
$$
(1-r)^2\sin h \cos h.
$$
The area of the remaining part can be computed by difference, as the area of the circular sector with vertexes $P_jP_{j+1}P_{j+2}$ and the triangle with the same vertexes. The result is
$$
\frac{a-\sin a}{2}.
$$
Let us now consider $B$. The triangle in $B$ (not isosceles) has the following structure: the two sides which meet at $O$ have length $1-r$ and $\cos(a/2) - (1-r)\cos h$, respectively, and form an angle of amplitude $\ell + h$; therefore its area is
$$
\frac12 (1-r) \big(\cos(a/2)-(1-r)\cos h\big) \sin(h+\ell).
$$
As already done for $A$, it is immediate to check that the remaining part in $B$ has area 
$$
\frac{b-\sin b}{2}.
$$
By summing up the contributions we find \eqref{ap1}. \end{proof}
\begin{remark}\label{maybe}
Notice that the formula above is valid also for $h= 0 $ or $\ell$, with the appropriate interpretation. As already noticed in Remarks \ref{remark} and \ref{rem-cluster}, when $h=0$, the cluster reduces to a single extremal arc. The formula above at 0 gives 
$$
F(r,0)=(1-r)^2\sin \ell \cos\ell + \frac{\ell - \sin \ell}{2};
$$
which is the area of the subdomain bounded by an extremal arc and the two segments joining its endpoints to the origin.
Similarly, when $h=\ell$ we have three extremal arcs, which is in accordance to 
$$F(r,\ell(r))=3F(r,0).$$
\end{remark}

The properties of $F$ are summarized in the following.
\begin{proposition}\label{prop5}
The first and second derivatives of $F$ with respect to the second variable are given by
\begin{align}
\frac{\partial F(r,h)}{\partial h} = &1+2(1-r)^2 \cos(2h)+2(1-r)\frac{\cos h}{\cos(a/2)}\,\left(2(1-r)^2\sin^2 h -1\right), \label{ap2}
\\
\frac{\partial^2 F(r,h)}{\partial h^2} = &  - 4 (1-r)^2 \sin(2h) + 2 (1-r)^5 \frac{\sin^3 h \cos^2 h}{\cos^3(a/2)}+\notag
\\
&  2 (1-r)\frac{\sin h}{\cos(a/2)} \left( 1-2(1-r)^2\sin^2 h + 3 (1-r)^2 \cos^2 h  \right),\label{ap3}
\end{align}
where $a=a(r,h)$ is the function introduced in \eqref{statement-a}. 
In particular, 
$$
\frac{\partial^2 F(r,h)}{\partial h^2}<0
$$
for any $h\in [0,\ell(r)]$, $\ell(r)$ being the function introduced in \eqref{statement-l}.
\end{proposition}
\begin{proof}
Throughout the proof $r$ is fixed, therefore we omit the dependence on it. In particular, $F$, $a$, and $b$, introduced in \eqref{ap1}, \eqref{statement-a}, \eqref{statement-b}, respectively, will be regarded as functions of the sole variable $h$, and their derivatives will be denoted simply by a prime.

We will use the following formulae, which can be deduced from $\tan(\ell/2)=\sqrt{4(1-r)^2-1}$ (cf. \eqref{statement-l}):
\begin{equation}\label{ap4}
\cos(\ell/2)=\frac{1}{2(1-r)}, \quad \sin(\ell/2)=\frac{\sqrt{4(1-r)^2-1}}{2(1-r)},
\end{equation}
and
\begin{equation}\label{ap5}
\cos\ell=\frac{1}{2(1-r)^2} - 1,\quad \sin\ell=\frac{\sqrt{4(1-r)^2-1}}{2(1-r)^2}.
\end{equation}
From the definition of $a$ and $b$, we have
\begin{equation}\label{ap6}
a'=2(1-r) \frac{\cos h}{\cos(a/2)},\quad 
b'=1 - (1-r) \frac{\cos h}{\cos(a/2)}.
\end{equation}
Differentiating $F$ and using \eqref{ap6} yields:
$$\begin{array}{l}
F'(h)=(1-r)^2 \cos(2h) +(1-r) (1-\cos a)\frac{\cos h}{\cos(a/2)} +\\
\left((1-r)\sin h-(1-r)^2 \sin h \frac{\cos h}{\cos(a/2)}\right) (1-r) \sin(h+\ell) +\\
(\cos(a/2)-(1-r)\cos h) (1-r) \cos(h+\ell) +(1-\cos b) \left(1-(1-r) \frac{\cos h}{\cos(a/2)}\right).
\end{array}$$
Using \eqref{statement-a}, \eqref{statement-b}, and \eqref{ap4}, we obtain
$$
\cos a= 1- 2(1-r)^2\sin^2 h,\quad \cos(a/2)=\frac{1-(1-r)^2\sin^2 h}{\cos(a/2)},
$$
and
$$
\begin{array}{c}
\cos b=\cos h \cos\left(\frac{\ell-a}{2}\right)-\sin h \sin\left(\frac{\ell -a}{2}\right)=\\
\frac{\cos(a/2)}{2(1-r)}\left[\cos h-\sin h \sqrt{4(1-r)^2-1} \right]+
\frac{\sin h}{2}\left[\cos h \sqrt{4(1-r)^2-1}
+ \sin h\right].
\end{array}
$$
These computations allow to simplify the expression above of $F'$ and to get \eqref{ap2}.

\medskip
Differentiating one more time \eqref{ap2} we get
\begin{align*}
F''(h)=& -8 (1-r)^2 \cos h \sin h + 2(1-r)^3 \sin h \cos^2 h \left[ 2(1-r)^2 \sin^2 h - 1\right]/\cos^3(a/2) +
\\
& 2 (1-r)\sin h\left[ 1-2(1-r)^2\sin^2 h + 4 (1-r)^2 \cos^2 h \right]/\cos(a/2).
\end{align*}
Finally, writing $2(1-r)^2\sin^2h - 1 = (1-r)^2\sin^2 h - \cos^2(a/2)$ and reordering the terms, we arrive at \eqref{ap3}.

\medskip
Let us now prove that $F''(h)<0$ when $h\in [0,\ell]$. To this aim, we write the second derivative as $F''=2(1-r)\sin h (A+B)$, with
\begin{align*}
A(r,h):=& -\frac{1}{9} (1-r) \cos h  + \frac{(1-r)^4\sin^2 h \cos^2 h}{\cos^3(a/2)}
\\
B(r,h):=&  -\frac{35}{9} (1-r)\cos h  + \frac{3(1-r)^2\cos^2 h - 2 (1-r)^2\sin^2 h + 1}{\cos(a/2)}.
\end{align*}
If we prove that $A$ and $B$ are negative, we are done.

Since $h\mapsto a(h)$ is increasing, both terms in $A$ are increasing. Therefore
$A(r,h)\leq A(r,\ell)$. Using \eqref{ap4} and \eqref{ap5} we get
$$A(r,\ell)=-\frac{1}{9}(1-r)\left[\frac{1}{2(1-r)^2}-1\right]+8(1-r)^7
\frac{4(1-r)^2-1}{4(1-r)^4}
\left[\frac{1}{2(1-r)^2}-1\right]^2.$$
This leads to look at the sign of
the polynomial $P(x)=-4x^4+3x^2-\frac{5}{9}$, with $x:=1-r$. Since the roots of $P$ 
are $1/\sqrt{3}$ and $\sqrt{5/12}$, $P$ is negative in $[\frac{1}{2},\frac{1}{\sqrt{3}}]$, we conclude that $A(r,h)\leq 0$.

\medskip
Let us look at $B(r,h)$. It has the same sign of
$$B_1(r,h)=-\frac{35}{9} (1-r)\cos h \cos(a/2) + 1 +\frac{(1-r)^2}{2}(5\cos(2h) +1).$$
Now, comparing their $\sin$, it is immediate that, for any $h$, $a/2\leq (1-r)h$.
Therefore,
$$B_1(r,h)\leq B_2(r,h)= -\frac{35}{9} (1-r)\cos h \cos((1-r)h) + 1 +\frac{(1-r)^2}{2}(5\cos(2h) +1).$$
We now compute the three first derivatives of $h\mapsto B_2(r,h)$. It comes, after
linearisation
$$\frac{d^3 B_2}{dh^3}=(1-r)\left[-\frac{35}{18}\left(r^3\sin(rh)+(2-r)^3\sin((2-r)h)\right)
+20(1-r)\sin(2h)\right].$$
Using $\sin(rh)\leq rh$, $\sin((2-r)h)\leq (2-r)h$ and $\sin(2h)\geq 4h/\pi$
we get 
$$\frac{d^3 B_2}{d h^3}\geq (1-r)h\left[-\frac{35}{18}\left(r^4+(2-r)^4\right)
+\frac{80}{\pi}(1-r)\right].$$
Since $r^4+(2-r)^4\leq 56/9$ and $1-r\geq 1/2$ we conclude that 
$\frac{d^3 B_2}{dh^3} \geq 0$ and then $\frac{d B_2}{dh}$ is convex in $h$,
moreover it vanishes at $0$. Thus, $\frac{d B_2}{dh}$ is either always positive or always
negative or negative and then positive (and this is actually the case).
In any case, we see that
$$B_2(r,h)\leq \max\left(B_2(r,0),B_2(r,\ell)\right).$$
Now we see that $B_2(r,0)=3(1-r)^2-\frac{35}{9}(1-r)+1 \leq -\frac{7}{36} <0$.

It remains to estimate $B_2(r,\ell)$. For that purpose,  we claim the following:
\begin{equation}\label{ap7}
\cos((1-r)\ell) \geq \frac{11}{5} - \frac{12}{5}(1-r) = \frac{12}{5} r - \frac{1}{5}.
\end{equation}
Recalling the relation \eqref{statement-l} between $\ell$ and $r$, the validity of \eqref{ap7} is related to the positivity of the auxiliary function
$$\psi(r):=\arccos\left(\frac{12}{5} r - \frac{1}{5}\right)-
2(1-r)\arctan\left(\sqrt{4(1-r)^2-1}\right).$$
The second derivative of $\psi$ reads
$$
\psi''(r)=-\frac{144(12 r-1)}{\left( 25-(12 r -1)^2\right)^{3/2}} + \frac{2}{(1-r)(4(1-r)^2-1)^{3/2}}
$$
and is negative in $[1-1/\sqrt{3},1/2]$. In particular $\psi$ is concave and 
$$\psi(r)\geq \min (\psi(1-1/\sqrt{3}),\psi(1/2))\geq 0.$$
This proves the claim.

We insert the estimate \eqref{ap7} in $B_2$ to get (we still use $x=1-r$):
$$B_2(r,\ell) \leq -\frac{7}{9}x\left(\frac{1}{2x^2}-1\right)(11-12x)+1+\frac{x^2}{2}
\left( 5\left(\frac{1}{2x^2}-1\right)^2-5\frac{4x^2-1}{4x^4} +1\right).$$
This leads to consider the polynomial
$$Q(x)=-\frac{19}{3}x^4+\frac{77}{9}\,x^3+\frac{2}{3}\,x^2-\frac{77}{18}\,x+\frac{5}{4}.$$
This polynomial is negative in $[1/2,1/\sqrt{3}]$, implying that $B_2$ is negative too.
\end{proof}

\begin{proof}[Proof of Theorem \ref{thm-density}.] We begin by noticing that if $N<N(r)$, then the class of admissible shapes is empty: assume by contradiction that there exists a Reuleaux polygon contained into the annulus $B_{1-r}(0)\setminus \overline{B}_r(0)$ with $M<2N(r)+1$ sides. Each arc of the boundary has length at most $\ell(r)$, therefore, imposing that the perimeter is $\pi$ and recalling the definition \eqref{def-N} of $N(r)$, we get
$$
\pi \leq M \ell(r) \leq (2N(r)-1) \ell(r)= \left(2 \left \lceil\frac{\pi}{2\ell(r)}-\frac12 \right \rceil-1 \right) \ell(r) < \pi,
$$ 
which is absurd.

Let now $N\geq N(r)$. The proof is divided into four steps.
\medskip

\textit{Step 1.} By Definition \ref{defrigid}, any shape that is not rigid can be modified, through an admissible
Blaschke deformation, to decrease the area. Thus, it remains to minimize the area
among rigid shapes. We have seen in Proposition \ref{propfund} that these rigid shapes
are composed of extremal arcs and clusters.

\medskip

\textit{Step 2.} Let us write an area formula for a rigid shape $\Omega$. Connecting with straight segments the origin to the vertexes, we split $\Omega$ into subdomains, which can be regrouped as triples of subdomains associated to clusters  {(see also Fig. \ref{fig-clu})} and subdomains associated to extremal arcs. According to the notation used in Remark \ref{rem-cluster}, all the subdomains can be regarded associated to clusters, allowing the parameters $h_i$ to vary in the closed interval $[0,\ell(r)]$, $i=1, \ldots, \tilde{m}$, for a suitable $\tilde{m}\in \mathbb N$. In view of Lemma \ref{lem-F} and Remark \ref{maybe}, we infer that the total area is
$$
|\Omega|=\sum_{i=1}^{\tilde{m}}F(r, h_i).
$$
Notice that the area does not explicitly depend on the relative position of the clusters (this dependence is enclosed into the relation among the lengths).
The necessary condition \eqref{nc2} gives a restriction on the possible values of $\tilde{m}$: since every $h_i$ is between $0$ and $\ell(r)$, we infer that $2h_i + \ell(r) \in [\ell(r), 3 \ell(r)]$, so that, summing over $i$ from $1$ to $\tilde{m}$, we get
\begin{equation}\label{range}
\frac{\pi}{3\ell(r)}\leq \tilde{m} \leq  \left \lfloor\frac{\pi}{\ell(r)}\right \rfloor.
\end{equation}
In particular, this implies that the number of sides of a rigid configuration cannot be arbitrarily large, but it is bounded by a quantity depending only on $r$. We infer that the sequence of minima $\{\A_N(r)\}_{N\geq N(r)}$ is constant after a finite number of values (depending on $r$). A priori, the first terms of the sequence could be different. In the next step we show that, actually, the sequence is constant in $N$.

\textit{Step 3.} Let us optimize the area when $\tilde{m}$ is fixed. In view of the previous step, we are led to minimize the function 
$$
\mathcal F(h_1,\ldots, h_{\tilde{m}}):=\sum_{i=1}^{\tilde{m}} F(r,h_i),
$$
over the set 
$$
\mathcal C:=\left\{ (h_1,\ldots, h_{\tilde{m}})\in [0,\ell(r)]^{\tilde{m}},\quad \sum_{i=1}^{\tilde{m}} [2h_i + \ell(r)] = \pi\right\}.
$$
In that way, we transform a geometric problem into an analytic one which might have
solutions that do not correspond to real geometric shapes. It turns out, as we will see
below, that the minimizer is unique and actually corresponds to a real body of constant width.

The set of constraints is the intersection of an hypercube and an hyperplane. In view of Proposition \ref{prop5}, the function $\mathcal F$ is strictly concave, therefore it attains a minimum on extremal points of $\mathcal C$. The extremal points of $\mathcal C$ lie on the edges of the hypercube, namely (up to relabeling) $h_1\in [0,\ell]$ and $h_2,\ldots, h_{\tilde{m}}\in \{0,\ell\}$. Without loss of generality, we may label the $h_i$s in such a way that $h_2,\ldots, h_q=0$ and $h_{q+1}, \ldots, h_{\tilde{m}}=\ell$. We claim the following facts:
\begin{itemize}
\item[(i)] for $\tilde{m}$ fixed, the extremal point of $\mathcal C$ is unique;
\item[(ii)] for a fixed $r$, the minimum of $\mathcal F$ does not depend on $\tilde{m}$ in the range \eqref{range}.
\end{itemize}
The case in which $r$ is the inradius of some regular Reuleaux polygon is trivial: in view of \eqref{nc2}, the parameter $h_1$ has to belong to $\{0,\ell(r)\}$ and, again by \eqref{nc2}, no matter how the sides are regrouped (one by one when $h=0$, three by three when $h=\ell(r)$), they are necessarily $2N(r)+1$, where $N(r)$ is the number introduced in \eqref{def-N}.

In all the other cases, $h_1$ lies necessary between $0$ and $\ell(r)$, strictly. A first consequence is that the number of sides is $3 + (q-1)+3(\tilde{m}-q)$. Since it is odd, we infer that $\tilde{m}$ is odd, too. In view of \eqref{nc2}, $q$ is given by
$$
q = \frac32 \tilde{m} - \frac{\pi}{2\ell(r)}+ \frac{h_1}{\ell(r)}.
$$
More precisely, taking into account that $h_1/\ell(r) \in (0,1)$, we get 
$$
q=q(\tilde{m}):=1+\left\lfloor \frac32 \tilde{m} - \frac{\pi}{2\ell(r)} \right\rfloor.
$$
Using again \eqref{nc2}, we infer that $h_1$ is given by
$$
h_1=\ell(r)(1-\delta),
$$
with
$$
\delta:= \frac32 \tilde{m} - \frac{\pi}{2\ell(r)} - \left\lfloor \frac32 \tilde{m} - \frac{\pi}{2\ell(r)}\right\rfloor \in (0,1).
$$
All in all, once fixed $\tilde{m}$, $h_1$ and $q$ are determined. This concludes the proof of (i).

Notice that if we replace $\tilde{m}$ by $\tilde{m}+2$ (as already noticed $\tilde{m}$ has to be odd), the value of $\delta$ does not change. This allows us to write $h_1$, without the dependence on $\tilde{m}$. Therefore, in order to prove (ii), it is enough to show that the number of sides of length $\ell(r)$ does not depend on $\tilde{m}$:
\begin{align*}
q& (\tilde{m})-1+ 3 (\tilde{m}-q(\tilde{m})) = 3 \tilde{m} - 2 q(\tilde{m})-1  = 3 \tilde{m} - 2 \left\lfloor \frac32 \tilde{m} - \frac{\pi}{2\ell(r)} \right\rfloor - 3 
\\ & = 3 (\tilde{m}-1) - 2 \left\lfloor \frac32 (\tilde{m}-1) - \left(\frac{\pi}{2\ell(r)} -  \frac32 \right) \right\rfloor = 2 \left( \left\lfloor  \frac{\pi}{2\ell(r)} -  \frac32 \right\rfloor+ 1\right). 
\end{align*}
Here we have used that $\tilde{m}-1$ is even, together with the equality $\left\lfloor k - x\right\rfloor= k -\left\lfloor x \right\rfloor - 1$, true for every $k\in \mathbb N$ and every $0<x<k$, $x\notin \mathbb N$. This proves (ii).

\medskip

\textit{Step 4.} In view of the previous step, we immediately get that the optimal shape associated to the inradius $r$ of a regular Reuleaux polygon, is the Reuleaux polygon itself, for every $N\geq N(r)$. When $r$ is not the inradius of a regular Reuleaux polygon, we have shown that, for every $N\geq N(r)$, the optimal configuration has a unique cluster and have all the other sides of length $\ell(r)$. As already underlined in Definition \ref{def-optimals}, these properties characterize the set $ \Omega(r)$, and the proof of the theorem is concluded. Note that $2 \left\lfloor \pi/(2\ell(r))- 3/2\right\rfloor+2$ (i.e. the number of sides of length $\ell(r)$ found in Step 3) is equal to $N(r)-2$, implying that the total number of sides of the optimal shape is $2N(r)+1$, as expected.
\end{proof}

We are now in a position to prove the main results, about the characterization of minimizers and the continuity of minima and minimizers with respect to $r$.
\begin{proof}[Proof of Theorem \ref{thm}]
In view of the density of the Reuleaux polygons in the class of constant width sets see \cite{BF} or \cite{Buc}, we infer that
$$
\A(r)=\inf_{ {N\geq 1}} \A_N(r).
$$
In view of Theorem \ref{thm-density}, we infer that the sequence $\A_N(r)$ is finite and constant after $N(r)$, so that $\A(r)=\inf_{ {N\geq 1}}\A_N(r)=|\Omega(r)|$. The other statements follow from the characterization of $\Omega(r)$ (see Definition \ref{def-optimals} and Proposition \ref{prop-clu}). 
\end{proof}

\begin{proof}[Proof of Proposition \ref{prop2}.] In view of Theorem \ref{thm}, its proof, and Definition \ref{def-optimals}, $\A(r)$ can be computed by dividing the optimal shape $\Omega(r)$ into $2N(r)+1$ subdomains, obtained by joining with segments the vertexes with the origin. The partition is made of subdomains associated to extremal arcs of length $\ell(r)$ and (possibly) to one triple associated to the cluster of parameter $h(r)$. According to \eqref{ap1}, the former have all area $F(r,0)$, the latter (when present) has area  $F(r,h(r))$. 
When $r=\r$ the partition is regular and
$$
\A(\r) = (2N+1)F(\r,0).
$$
In all the other cases, namely when $r\in (r_{_{2N-1}},\r)$, we have
$$
\A(r)=(2N-2)F(r,0) + F(r,h(r)).
$$
In the open interval $(r_{_{2N-1}}, \r)$ the functions $F(r,0)$, $h(r)$, and $F(r,h(r))$ are continuous, therefore $\A(r)$ is continuous too. In the limit as $r \searrow r_{_{2N-1}}$, we have $h(r)\to 0$, so that 
$$
\lim_{r\searrow r_{_{2N-1}}}\A(r)= (2N-2)F(r_{_{2N-1}},0) + F(r_{_{2N-1}},0) = \A(r_{_{2N-1}});
$$
similarly, when $r\nearrow \r$, we have $h(r)\to\ell(\r)$ and $F(r,\ell(r))\to 3 F(\r,0)$, thus 
$$
\lim_{r\nearrow \r}\A(r)=(2N-2)F(\r,0) + 3F(\r,0)=(2N+1)F(\r,0) = \A(\r).
$$
Therefore, $\A$ is continuous in each closed interval $[r_{_{2N-1}},\r]$. This concludes the proof of the continuity of $\A$.

Let us now consider the optimal shapes. We choose the following orientation: for regular Reuleaux polygons, we take one of the vertexes aligned vertically with the origin, above it; in all the other cases, we choose the point $P_k$ of the cluster (see Definition \ref{def-clu}) aligned vertically with the origin, below it (see also Fig. \ref{fig-os}). By construction, the position of the vertexes varies continuously with respect to $r$, so that the optimal shapes vary continuously with respect to the Hausdorff convergence. 
\end{proof}

 {\begin{remark}
We could also be interested in {\it the dual problem}, namely to {\it maximize} the inradius among shapes of constant width and fixed area. If we check that the function $r\mapsto \A(r)$ is strictly increasing
(what is clearly true numerically), then an easy argument shows that the domain $\Omega(r)$ also
solves this dual problem.
\end{remark}}

\bigskip

{\bf Acknowledgements}: 
The authors want to thank G\'erard Philippin for stimulating discussions  {and the two anonymous referees
for their very interesting suggestions}.
This work was partially supported by the project ANR-18-CE40-0013 SHAPO financed by the French Agence Nationale de la Recherche (ANR).
IL acknowledges the Dipartimento di Matematica - Universit\`a di Pisa for the hospitality.

\medskip

Antoine \textsc{Henrot}, Universit\'e de Lorraine CNRS, IECL, F-54000 Nancy, France email: \texttt{antoine.henrot@univ-lorraine.fr} 

Ilaria \textsc{Lucardesi}, Universit\'e de Lorraine CNRS, IECL, F-54000 Nancy, France email: \texttt{ilaria.lucardesi@univ-lorraine.fr}

\end{document}